\documentclass{article}
\usepackage[english]{babel}
\usepackage[utf8]{inputenc}
\usepackage{amsmath,amssymb,graphicx,bbm,xcolor,latexsym}
\usepackage[numbers]{natbib}
\usepackage[hidelinks]{hyperref}
\hypersetup{
    colorlinks,
    linkcolor={red!50!black},
    citecolor={blue!50!black},
    urlcolor={blue!80!black}
}
\newcommand{\nocomma}{}

\newtheorem{example}{Example}

\newenvironment{remark}{\noindent\textbf{Remark.}}{\medskip}
\newtheorem{theorem}{Theorem}
\newenvironment{proof}{\noindent\textit{Proof.}}{\hspace*{\fill}$\Box$\medskip}
\newenvironment{proof*}[1]{\noindent\textbf{#1\ }}{\hspace*{\fill}$\Box$\medskip}
\input{tcilatex}

\begin{document}

\title{\textsc{Some remarks on the\\
coincidence set\\
for the {S}ignorini problem}}

\author{{M}iguel de {B}enito {D}elgado \and {J}esus {I}ldefonso {D}íaz }
\maketitle

\begin{center}
\emph{From two generations of students \\[0pt]
to a fabulous teacher and friend now in his retirement: \\[0pt]
Baldomero Rubio}
\end{center}

\bigskip

\begin{abstract}
We study some properties of the coincidence set for the
\textit{boundary Signorini problem}, improving some results from
previous works by the second author and collaborators. Among other
new results, we show here that the convexity assumption on the
domain made previously in the literature on the location of
the coincidence set can be avoided under suitable alternative conditions
on the data.
\end{abstract}

\label{page:t} \thispagestyle{plain}



\newpage

\section{Introduction}
\label{sec:introduction}

In the classical Signorini problem of linear elasticity 
\cite{signorini_questioni_1959}, or \textit{boundary obstacle problem},
an isotropic, homogeneous and linearly hyperelastic material rests in
equilibrium over a rigid foundation. Because the contact
zone is an unknown of the problem, estimates on its location and size are of
interest in the study of the properties of solutions. In the scalar
Signorini problem displacements take place along one dimension only and the
equation of conservation of momentum is reduced to Poisson's equation. The 
simplified model we shall consider in this paper is the following:

\begin{equation*}
\left\{
\begin{array}{rl}
-\Delta u=f & \text{ in }\Omega , \\
\begin{array}{c}
u\geqslant \psi ,\text{ }\partial _{\nu }u\geqslant g \\
(u-\psi )(\partial _{\nu }u-g)=0%
\end{array}
& \text{ on }\Gamma .%
\end{array}%
\right.
\end{equation*}

Notice that, although the prototypical model for boundary obstacle problems
is the one in elasticity theory, other related models with similar boundary
conditions are found for instance in semipermeable membranes (the so called
\textit{parabolic Signorini problem}) or stochastic control (with
fractional powers of the Laplacian). See the last series of Remarks
at the end of the paper.

We recall that the weak mathematical formulation of the model (what we will
refer to as \textbf{Problem 1$\label{prob0}$}) is the following: given
an open, bounded set $\Omega \subset \mathbbm{R}^{N}$ with Lipschitz boundary $\Gamma
=\partial \Omega $ and functions $\psi \in H^{1/2}(\Gamma ), g \in H^{-1/2}(\Gamma)$
and $f\in L^{2}(\Omega )$, find $u \in K_{\psi }:=\left\{ v\in H^{1}(\Omega): v\geqslant
\psi \text{ on }\Gamma \right\} $ such that
\begin{equation}
a_{0}(u,v-u)\geqslant F(v-u)\text{ for all }v \in K_{\psi }
\label{eq:var-ineq0}
\end{equation}%
where
\begin{equation}
a_{0}(u,v):=\int_{\Omega }\nabla u\nabla v\mathrm{d}x\text{ and }%
F(v):=\int_{\Omega }fv\mathrm{d}x+\langle g,v\rangle _{H^{-1/2}(\Gamma)
\times H^{1/2}(\Gamma )}.  \label{defini F}
\end{equation}

Since the bilinear form is not coercive some additional conditions on the
data must be introduced. In particular, we must assume the compatibility
condition
\begin{equation}
\int_{\Omega }f\mathrm{d}x+\langle g,1\rangle _{H^{-1/2}(\Gamma )\times
H^{1/2}(\Gamma )} \leq 0.  \label{eq:compatibility-condition}
\end{equation}
Notice that (\ref{eq:compatibility-condition}) is the
one-dimensional equivalent of the general condition for vectorial
formulations of the problem considered initially by Fichera {\cite[p. 81]%
{Fichera:1970wf}}, although there it is given in the equivalent form:
\begin{equation*}
\int_{\Omega }\mathbf{f}\cdot \mathbf{r}\mathrm{d}x+\int_{\Gamma }\mathbf{g}%
\cdot \mathbf{r}\mathrm{d}s_{x} \leq 0
\end{equation*}
for every rigid and admissible displacement $\mathbf{r}$, with equality if
and only if $\mathbf{-r}$ is also admissible, i.e. the cone of displacements
moving the body away from the obstacle. Equivalently, condition (\ref%
{eq:compatibility-condition}) means that rigid displacements separating the
body from the obstacle increase the elastic energy.

Existence and uniqueness of solution of a general class of problems
including Problem 1 follow from {\cite[Theorem 5.1]{Lions:1967jv}}
which proves the result for general non-symmetric semicoercive
bilinear forms, with uniqueness up to a member of a given subset of
the rigid displacements $\mathbf{r}$ satisfying the condition
$F(\mathbf{r})=0$ (where $F$ was defined in (\ref{defini F})). In
our case, since the unknown is scalar we obtain uniqueness once the
compatibility condition is assumed.

The {\bfseries{coincidence set}} for a solution $u \in H^{1}(\Omega )$ is
defined as the complement of the open set $\left\{ x \in \Gamma :u(x)>\psi (x)%
\text{ in the sense of }H^{1/2}(\Gamma )\right\} $, i.e.
\begin{equation}
I_{\psi }:=\{x\in \Gamma :u(x)=\psi (x)\}.  \label{eqdef:coincidence-set}
\end{equation}

Observe that it is not justified to require {\itshape{a priori}}
$g\in H^{1/2}(\Omega )$ since $a_{0}$ is not coercive and the
solution may fail to be in $H^{2}(\Omega )$ in very simple cases.
See, e.g., {\cite[Theorem I.10]{Brezis:1972tt} and \cite[p.
617]{Kinderlehrer:1981rs}} for some classical counterexamples of
cases in which $u\not\in H^{2}(\Omega)$, as well as the results
presented in \cite{schatzmann_problemes_1973} and \cite{caffarelli_further_1979}.

A common recourse against the lack of coercivity of the bilinear form is to
replace the equation by a new one by introducing an additional regularizing
term $\alpha u$ with $\alpha >0$ which makes the proofs of existence and
uniqueness straightforward. This is done e.g. in {\cite[Theorem I.10]%
{Brezis:1972tt}}. In this case the corresponding problem (which we shall
refer to as \textbf{Problem $\mathbf{\alpha}$}) involves the PDE $-\Delta
u+\alpha u=f$ which leads to the bilinear form
\begin{equation*}
a_{\alpha }(u,v):=\int_{\Omega }\nabla u\nabla v\mathrm{d}x+\alpha
\int_{\Omega }uv\mathrm{d}x\text{.}
\end{equation*}

Here, the additional linear term $\alpha u$ makes $a_{\alpha }$ coercive
even in the case of no Dirichlet boundary conditions and allows for a standard
proof of existence and uniqueness applying Stampacchia's theorem,
{\cite[Theorem 2.1]{Lions:1967jv}}. Coercivity also allows the use of Brézis'
regularity result {\cite[Theorem I.10] {Brezis:1972tt}} stating
$u\in H^{2}(\Omega )$ whenever $f\in L^{2}(\Omega )$, making the choice
$g\in H^{1/2}(\Gamma )$ adequate. We also note that
under additional assumptions on the data, the solution is in $L^{\infty}(\Omega )$
(see, e.g. \cite[Theorem I.10]{Brezis:1972tt} and \cite[§5]{Kinderlehrer:2000vi}).
See also the monograph \cite{petrosyan_regularity_2012}.

Concerning the estimates on the spatial location of the coincidence set we
recall that after Friedman's pioneering paper \cite{Friedman:1967bn}, the
first explicit estimates on the location were given in {\cite%
{Diaz:1980lf,DiazDiaz:1988wp}} for Problem $\alpha$ with $g=\psi =0$%
, under the geometric requirement that $\Omega$ be convex and by assuming
that the external force be negative near a sufficiently large part of the
boundary $\Gamma$.

In Section 2 we extend the conclusion of the above mentioned papers to
Problem 1 (see Theorem \ref{th:coincidence-convex}) while also
relaxing their assumptions. Section \ref{sec:dispense} is devoted to the
study of location estimates of the coincidence set when the convexity
assumption on $\Omega $ is not made. We provide Example \ref{exa:ring}
in which the coincidence set is totally identified on a non-convex $\Omega$.
Finally, by working with a suitable change of coordinates, we prove the main
result of this paper (Theorem \ref{th:coincidence-general}) in which we obtain
some location estimates of the coincidence set without any geometrical assumption
on $\Omega$ but instead some regularity condition on $\partial \Omega $.

\section{Location estimates for Problem 1}
\label{sec:example}

As already mentioned, in {\cite[Theorem 2]{DiazDiaz:1988wp}} the basic
geometrical assumption made for the study of Problem $\alpha$ with $%
g=\psi =0$ is that the domain $\Omega$ must be convex. In this section we
first improve on the aforementioned result by considering Problem 1
(i.e., without any regularization term) for non necessarily vanishing data $%
g$ and $\psi $. Moreover we shall assume convexity only for parts of $\Omega
$ near the boundary on which a suitable balance between the external force,
the obstacle and the boundary flux becomes negative. We shall also require $%
C^{3} $ boundary in order to have a tubular neighborhood of $\partial \Omega
$ with a $C^{2}$ parametrization given by $x=x(\xi ,s)=\xi +s\nu (\xi )$.

\begin{theorem}
\label{th:coincidence-convex}Let $\Omega $ be an open set in $\mathbbm{R}^{N}$
and assume that the data $f,g$ and $\psi$ lead to a unique solution $u\in
H^{1}(\Omega )\cap L^{\infty }(\Omega )$ of Problem 1.
Assume that there exist $\epsilon, \delta \geq 0$, a subset $\Gamma_{\epsilon, \delta}
\subset \Gamma = \partial \Omega$ of class $C^{3}$, and a tubular semi-neighborhood $V_{\rho }^{-}$
of $\Gamma_{\epsilon, \delta }$ for some $\rho >0$ with
\begin{equation}
\rho \text{ large enough and }V_{\rho }^{-}\subset \overline{\Omega },
\label{large enough}
\end{equation}
such that if $\Psi \in H^{2}(V_{\rho }^{-})$ is a nonnegative extension
of $\psi $ to $V_{\rho }^{-}$ (i.e. $\Psi = \psi $ on
$\Gamma_{\epsilon ,\delta }$) then one has
\begin{equation*}
f+\Delta \Psi \leq -\epsilon \text{ on } V_{\rho}^{-},
\end{equation*}
and
\begin{equation}
g-\partial _{\nu }\Psi \leqslant -\delta \text{ \ on \ }\Gamma _{\epsilon
,\delta }.  \label{Balance negative}
\end{equation}
Then, if $\epsilon > 0$ and
\begin{equation}
\Omega \cap V_{\rho }^{-}\text{ is a convex set,}  \label{convexity}
\end{equation}
we have the location estimate $\Gamma _{\epsilon ,\delta }\subset
I_{\psi }$ on the coincidence set of $u$.
\end{theorem}

We note that, in the case in which $f+\Delta \Psi =0$, by assuming the
coincidence set $I_{\psi }$ in the class of regular subsets of $\partial
\Omega $, a necessary condition in order to have a positively measured
coincidence set is that $\int_{I_{\psi }}g-\partial _{\nu }\Psi \leq 0$
(see e.g. \cite{Friedman:1967bn}). So, in some sense, Theorem 
\ref{th:coincidence-convex} shows that a pointwise balance estimate
(\ref{Balance negative}) on a \textit{good} part of $\partial \Omega $ is
enough to identify where the coincidence is taking place.

\bigskip

\begin{proof}
We first prove the result for the solutions {$\tilde{u}_{\alpha }$} of
Problem $\alpha$, for any $\alpha >0.$ Notice that it is enough to
prove the conclusion for data, $\psi =0,\widetilde{f}:=f+\Delta \Psi -$ $%
\alpha \Psi $ and $\tilde{g}:=g-\partial _{\nu }\Psi $. Indeed, we let $%
\tilde{u}$ be the solution of Problem 1 under these conditions,
and readily see that $u:=\tilde{u}+\Psi $ solves the problem with data $%
f,\psi ,g$. Let $x_{0}\in \Gamma _{\epsilon ,\delta }$ and, for $\rho >0$,
let $V_{\rho }^{-}$ be the tubular semi-neighborhood of $\Gamma _{\epsilon
,\delta }$ defined by the $C^{2}$ parametrization
\begin{equation*}
x=x(\xi ,s)=\xi +s\nu (\xi )\text{, for }\xi \in \Gamma _{\epsilon ,\delta
},s\in (-\rho ,0).
\end{equation*}%
Let $\Vert \widetilde{u}\Vert _{\infty }\leqslant M$. Define $R:=\min \{%
\mathop{\rm dist}\nolimits(x_{0},\Gamma \backslash \Gamma _{\epsilon ,\delta
}),\rho ,MN\}$ and consider the subset $D:=\Omega \cap B(x_{0},R).$ Define $%
\partial _{1}D:=\partial D\backslash \Gamma $ and $\partial _{2}D:=\partial
D\cap \Gamma \subset \Gamma _{\epsilon ,\delta }$ as in Figure
\ref{fig:proof-coincidence-convex}.

\begin{figure}[h]
\centering
\raisebox{-0.5\height}{\includegraphics{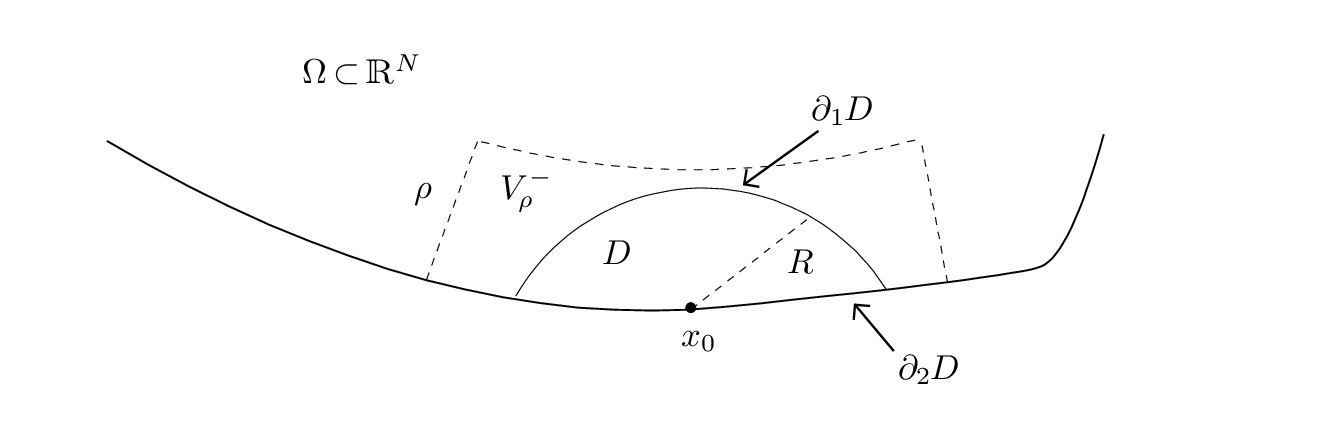}}
\caption{The setting in the proof of Theorem \protect\ref%
{th:coincidence-convex}.}
\label{fig:proof-coincidence-convex}
\end{figure}

For some $c > 0$ define now in $D$ the function $\overline{u}_{\alpha}
(x)=c(2N)^{-1}|x-x_{0}|^{2}$. In $D$ we have
\begin{equation*}
-\Delta \overline{u}_{\alpha }+\alpha \overline{u}_{\alpha }=-c+\frac{\alpha
c}{2N}|x-x_{0}|^{2}\geqslant - \epsilon \geq \widetilde{f}\text{ in }D,
\end{equation*}
assuming
\begin{equation*}
0<c\leq \epsilon .
\end{equation*}
Additionally, on $\partial _{1}D$, from assumption (\ref{large enough}) we
have
\begin{equation*}
\overline{u}_{\alpha }\geqslant \frac{c}{2N}R^{2}\geq M\geqslant \tilde{u}%
_{\alpha }
\end{equation*}%
which holds if
\begin{equation*}
R\geq \frac{2NM}{c}.
\end{equation*}
Moreover, since by construction $\overline{u}_{\alpha }$ is non-negative, on
the subset of the coincidence set $I^{\alpha }:=I_{0}^{\alpha }\cap \partial
_{2}D$ we have $\overline{u}_{\alpha }\geqslant \tilde{u}_{\alpha }$ a.e. $%
I^{\alpha }$ and thus a.e. on all of $\partial _{1}D\cup I^{\alpha }$. Now,
the Signorini conditions imply that it has to be $\partial _{\nu }\widetilde{%
u}_{\alpha }=\tilde{g}$ over $\partial _{2}D-I^{\alpha }$, and, on the other
hand we have
\begin{equation}
\partial _{\nu }\overline{u}_{\alpha }(\xi )=c N^{-1}|\xi -x_{0}|\cos (\nu
(\xi ),\xi -x_{0})\geqslant -\delta \geqslant \tilde{g}\text{ on }\partial
_{2}D  \label{Neumann}
\end{equation}
where we used the convexity of $\Omega \cap V_{\rho }^{-}$ (positivity of
the cosine) in the first inequality. Then, by the comparison principle
applied to the associated problem on the set $D$, with Signorini boundary
conditions on $\partial _{2}D$ and with Dirichlet conditions on 
$\partial_{1}D$ (see, e.g. \cite{Brezis:1972tt}) we deduce finally
that $\tilde{u}_{\alpha }\leqslant \overline{u}_{\alpha }$ in $D$ and that
the same inequality holds for the traces, that is:
\begin{equation*}
0\leqslant \tilde{u}_{\alpha }\leqslant \frac{c}{2N}|\xi -x_{0}|^{2},
\text{ over }\Gamma \cap B(x_{0}, R)
\text{ (in the sense of }H^{1/2}(\Gamma )\text{)}.
\end{equation*}
Letting $\xi \rightarrow x_{0}$ we conclude that for every $\alpha >0$,
$\tilde{u}_{\alpha }(x)=0$ a.e. in $\Gamma_{\epsilon, \delta}$ uniformly
(since the estimate on the location of $I^{\alpha }:=I_{0}^{\alpha }\cap
\partial _{2}D${\ is independent of }$\alpha $).

\noindent\textit{Final step.} For $\alpha \rightarrow 0$ we let $\tilde{u}_{\alpha }$
be the solutions with data $\widetilde{f}=0$, $\widetilde{\psi }=0$ and $%
\tilde{g}$ of Problem $\alpha$. The regularity result mentioned above implies
that we have $\Vert \tilde{u}_{\alpha }\Vert _{\infty}\leqslant M$
uniformly on $\alpha$. Then, by well known results we have $%
\tilde{u}_{\alpha }\rightarrow \tilde{u}_{0}$ strongly in $H^{1}(\Omega )$
and therefore $\tilde{u}_{0}=0$ on $\Gamma_{\epsilon ,\delta }$.
\end{proof}

\bigskip

\begin{remark}
In fact we can also consider the case $\epsilon =0$
and $\delta =0$ under the assumptions of Theorem \ref{th:coincidence-convex}.
Indeed, assume for the moment that we can construct, for some $c>0$, a function
$w\in H^{2}(D)$ satisfying

\begin{equation*}
\left\{
\begin{array}{lll}
-\Delta w\geq \widehat{\epsilon } & \text{in} & D, \\
w\geqslant 0 & \text{in} & \overline{D}, \\
w=0 & \text{on} & \partial _{1}D, \\
\partial _{\nu }w=-\widehat{\delta } & \text{on} & \partial _{2}D.%
\end{array}%
\right.
\end{equation*}%
Then, the function $\widehat{v}(x):={\tilde{u}_{\alpha }}(x)-w(x)$ is
a solution of Problem $\alpha$ for data $\psi =0,
\widehat{f} := f+\Delta \Psi -$ $\alpha \Psi $ $+\Delta w-\alpha w$ and
$\widehat{g}:= \widetilde{g}-\widehat{\delta }$.
Taking $w(x)=w(s\nu )=\varphi (s)$ and using the expression of the Laplacian
on $V_{\rho }^{-}$ (see, e.g. \cite{gilbarg_elliptic_2001} and
{\cite[§4.3.5, p. 62]{Sperb:1981tx}}), the construction of such a $w$ is
reduced to finding $\varphi (s)$ such that

\begin{equation*}
\left\{
\begin{array}{rll}
-\varphi ^{\prime \prime }(s)-(N-1)H(\xi ,s)\varphi ^{\prime }(s) &
\geqslant  & c, \\
\varphi (s) & \geqslant  & 0 \\
\varphi (0) & = & 0, \\
\varphi ^{\prime }(0) & = & -\delta ,%
\end{array}%
\right.
\end{equation*}%
for $s\in (-R,0)$, where $H(\xi ,s)$ denotes the mean curvature of the
hypersurface, to a distance $s$ of $\Gamma _{\epsilon ,\delta }$, i.e. at
points $x=x(\xi ,s)=\xi +s\nu (\xi ),\ \xi \in \Gamma _{\epsilon ,\delta }$.
When $\Omega \cap V_{\rho }^{-}$ is a convex set, as required in (\ref{convexity}),
we have that $H(\xi ,s)$ is non-negative and bounded for $R$ small enough depending
on this convex part of the boundary and thus $w$ can be made explicit. 
\end{remark}

\section{Sharper estimates and further remarks}
\label{sec:dispense}

One of the main goals of this section is to obtain some sharper location
estimates on the coincidence set and to extend the previous results
while dispensing with the convexity condition (\ref{convexity}) on the tubular
semi-neighborhood of $\Gamma _{\epsilon, \delta }$. We start by showing in a concrete
example that this goal is not impossible.

\begin{example}
\label{exa:ring}Let $0<R_{0}<R_{1}$ and define the open ring
\begin{equation*}
\Omega :=\{x\in B_{R_{1}}(0)\backslash \overline{B}_{R_{0}}(0)\}
\end{equation*}%
with inner boundary $\Gamma _{0}:=\partial B_{R_{0}}(0)$ and outer
boundary $\Gamma _{1}:=\partial B_{R_{1}}(0)$. Let $\varepsilon >0$,
$R_{\varepsilon }\in (R_{0},R_{1}]$ and
$f\in L^{2}(\Omega )\cap L^{\infty}(\Omega )$ be such that
\begin{equation*}
f(x)\leqslant -\varepsilon \text{\quad a.e. in\quad }\Omega _{\varepsilon
}:=B_{R_{\varepsilon }}(0)\backslash B_{R_{0}}(0),
\end{equation*}%
and consider the special formulation of the Signorini problem
\begin{equation*}
\left\{
\begin{array}{rlll}
-\Delta u & = & f & \text{ in }\Omega , \\
u & = & 0 & \text{ on }\Gamma _{1}%
\end{array}%
\right. \text{ and }\left\{
\begin{array}{rlll}
\partial _{\nu }u & \geqslant & 0 & \text{ on }\Gamma _{0}, \\
u & \geqslant & 0 & \text{ on }\Gamma _{0}, \\
u\partial _{\nu }u & = & 0 & \text{ on }\Gamma _{0}.%
\end{array}%
\right.
\end{equation*}
\end{example}

Notice that we now do not need the compatibility condition since we have
Dirichlet conditions on $\Gamma _{1}$.

We claim that the coincidence set is the whole $\Gamma_0$, i.e. 
$u_{|\Gamma _{0}}\equiv 0$.

\begin{figure}[h]
\centering
\raisebox{-0.5\height}{\includegraphics{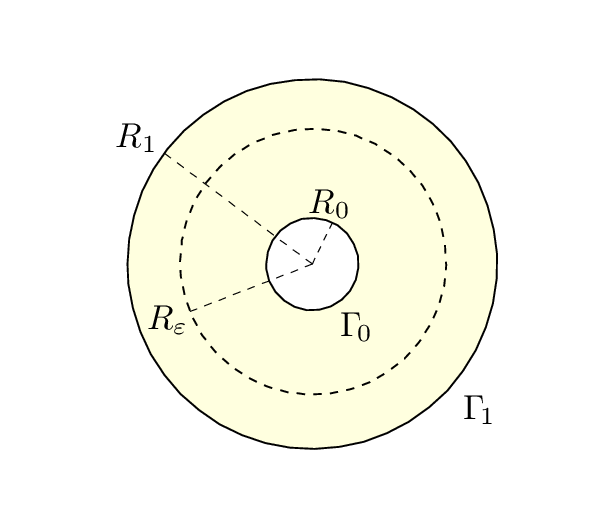}}
\caption{The contact set for a ring is the whole boundary.}
\end{figure}
To see this we apply again the comparison principle for the associated
Problem $\alpha$ for any $\alpha > 0 $ and then
we pass to the limit $\alpha \rightarrow 0$ as in Theorem \ref{th:coincidence-convex}.
We define over the ring $\Omega _{\varepsilon }$, the function
\begin{equation*}
\overline{u}(x):=c(|x|-R_{0})^{2}
\end{equation*}
for some constant $c>0$ to be determined later. Then, writing $r$ for $|x|$
we have
\begin{eqnarray*}
-\Delta \overline{u} &=&-\frac{d^{2}\overline{u}}{dr^{2}}-\frac{(N-1)}{r}%
\frac{d\overline{u}}{dr} \\
&=&-2c-\frac{(N-1)}{r}2c(r-R_{0}) \\
&\geqslant &-2c\left( 1+\frac{N-1}{R_{0}}(R_{\varepsilon }-R_{0})\right) .
\end{eqnarray*}
Consequently $-\Delta \overline{u}\geqslant -\varepsilon \geqslant f$ on $%
\Omega _{\varepsilon }$ whenever $c\leqslant \frac{\varepsilon }{2}\left( 1+%
\frac{N-1}{R_{0}}(R_{\varepsilon }-R_{0})\right) ^{-1}$. For instance, we
may take $c:=\varepsilon R_{0}/(4NR_{1})$. We set $\overline{f}:=-\Delta
\overline{u}$ in $\Omega _{\varepsilon }$ and by construction $\overline{f}%
\geqslant f$. Let $M:=\Vert u\Vert _{\infty ,\Omega }$. In order to apply
the comparison lemma, we need the condition $\overline{u}\geqslant M$ on $%
\Gamma _{\varepsilon }:=\partial B_{R_{\varepsilon }}(0)$, that is: $%
c(R_{\varepsilon }-R_{0})^{2}\geqslant M$ or, equivalently, $R_{\varepsilon
}\geqslant R_{0}+\sqrt{M/c}=R_{0}+2(MNR_{1}\varepsilon
^{-1}R_{0}^{-1})^{1/2} $ and this is possible for large enough values of $%
R_{1}$. Furthermore, on $I=\{x\in \Gamma _{0}:u(x)=0\}$ we have $\overline{u}%
\geqslant u$ too, since by construction $\overline{u}$ is non-negative. Thus
$\overline{u}\geqslant u $ on $\Gamma _{\varepsilon }\cup I$ and on the
complement $\Gamma _{0}\backslash I$, where $u>0$, the Signorini condition
implies $\partial _{\nu }u=0$ and also $\partial _{\nu }\overline{u}%
=-c(|x|-R_{0})x=0$. Applying the comparison principle we deduce $u\leqslant
\overline{u}$ on $\overline{\Omega }_{\varepsilon }$, and in particular $u=0$
over $\Gamma _{0} $.

\bigskip

The following result improves Theorem \ref{th:coincidence-convex}
for the case $\epsilon >0$ and $\delta >0$.

\begin{theorem}
Assume the conditions of Theorem \ref{th:coincidence-convex} except
condition (\ref{convexity}), for some fixed $\epsilon >0$ and $\delta >0$.
Then we obtain the location estimate
$\Gamma _{\epsilon ,\delta }\subset I_{\psi }$ on the coincidence set of $u$.
\end{theorem}

\begin{proof}
We follow the same arguments of the proof of Theorem \ref{th:coincidence-convex},
but instead of condition (\ref{convexity}) we use the fact that we can assume that
\begin{equation}
\inf_{x_{0}\in \Gamma _{\epsilon ,\delta },x\in V_{\rho }^{-}}\{\cos
(x-x_{0},\nu (x_{0})\}\geq -\theta _{0}(\delta ,\rho )
\end{equation}%
for some continuous function $\theta _{0}(\delta ,\rho )\geq 0$ (as a matter
of fact $\theta _{0}(\delta ,\rho )\in \lbrack -1,1]$). Then, in (\ref{Neumann})
we argue instead with
\begin{equation*}
\partial _{\nu }\overline{u}_{\alpha }(\xi )=cN^{-1}|\xi -x_{0}|\cos (\nu
(\xi ),\xi -x_{0})\geqslant -cN^{-1}R\theta _{0}(\delta ,\rho )\geqslant
-\delta \geq \tilde{g}\text{ on }\partial _{2}D
\end{equation*}
which holds by taking
\begin{equation*}
c N^{-1}\theta _{0}(\delta ,\rho )R\leq \delta ,
\end{equation*}%
i.e.
\begin{equation*}
c\leq \frac{\delta N}{\theta _{0}(\delta ,\rho )R}.
\end{equation*}
Therefore it is enough to take
\begin{equation*}
c\leq \min (\epsilon ,\frac{\delta N}{\theta _{0}(\delta ,\rho )R})
\end{equation*}%
and $R\geq \frac{2NM}{c}.$
\end{proof}

\bigskip

\begin{remark}
Notice that when the tubular semi-neighborhood of $\Gamma
_{\epsilon ,\delta }$ is convex then $\theta _{0}(\delta ,\rho )=0$. We
conjecture that, under suitable additional conditions, it should be possible
to dispense with at least one of the assumptions $\epsilon >0$ or $\delta >0$
on the semi-neighborhood of $\Gamma _{\epsilon ,\delta }$. At present it seems
that this fact remains an open problem in the literature. One could try to argue
as in the previous Remark in order to extend the result to the case $\epsilon =0$
but with $\delta >0$. However, it is not entirely clear how to construct
the function $w$ without condition (\ref{convexity}).
Notice that in the special case in which we assume that all the mean
curvatures  $H(\xi ,s)$ are constant and equal to $H$, we may actually
solve the auxiliary equation of the above Remark without requiring
$H\geqslant 0$, but under suitable choices of the interval
of definition of such functions. Indeed, the exact solution is given by
$\varphi (s)=\varphi _{h}(s)+\varphi _{p}(s)$. For the general solution of
the homogeneous equation $\Phi ^{\prime }=A\Phi $ with $\Phi =(\varphi
,\varphi ^{\prime })$, $A=\left(
\begin{array}{cc}
0 & 1 \\
0 & -b%
\end{array}%
\right) $ and $b=(N-1)H$ one has $\Phi (s)=\mathrm{e}^{As}\vec{\alpha}$ with
$\mathrm{e}^{As}=\left(
\begin{array}{cc}
1 & \frac{-1}{b}\mathrm{e}^{-bs} \\
0 & \mathrm{e}^{-bs}%
\end{array}%
\right) ,\vec{\alpha}\in \mathbbm{R}^{2}$. Therefore:
\begin{equation*}
\varphi _{h}(s)=\alpha _{1}-\frac{\alpha _{2}}{b}\mathrm{e}^{-bs}.
\end{equation*}%
For the particular solution of the inhomogeneous equation we find with the
Ansatz $\varphi _{p}(s)=\beta _{1}s^{2}+\beta _{2}s$:
\begin{equation*}
\varphi _{p}(s)=\frac{-c}{b}s.
\end{equation*}%
Introducing the boundary conditions we arrive at $\alpha _{1}=\alpha
_{2}/b=(c/b-\delta )/b$ and
\begin{equation*}
\varphi (s)=\left( \frac{c}{b^{2}}-\frac{\delta }{b}\right) (1-\mathrm{e}%
^{-bs})-\frac{c}{b}s,\quad s\leqslant 0.
\end{equation*}%
Finally we have $\varphi (s)\geqslant 0$ over some interval $(-\varepsilon ,0)$
because $\varphi ^{\prime }$ is continuous and negative at zero, therefore
in an interval around it, and $\varphi (0)=0$, meaning that the function
decreases to zero from the left.
\end{remark}

Our next goal is to improve the location estimates. In order to achieve this
we shall not use a function of the Euclidean norm as local supersolution,
but a differentiable extension of the intrinsic distance over the
manifold $\partial \Omega $. The gradient is then tangent at every point, $%
\partial _{\nu }\tilde{v}=0$ and we may build simple supersolutions. Let $%
l(\gamma _{a\nocomma b})$ denote the length of a piecewise $C^{1}$ curve $%
\gamma _{a\nocomma b}$ joining two points $a,b\in \Gamma $. Fix a point $%
x_{0}\in \Gamma $ and an open neighborhood $\Gamma _{0}$ of $x_{0}$ in $%
\Gamma $ whose closure is the graph of a Lipschitz map $\varphi :\overline{U}%
\rightarrow \mathbbm{R}$. Define the {%
\textbf{intrinsic distance to $\mathbf{x_0}$ over $\mathbf{\Gamma}$}} as
\begin{equation*}
d_{0}(x):=\inf \{l(\gamma _{x_{0}\nocomma x}):\gamma \in C^{1}([0,1],\Gamma
_{0}),\gamma (0)=x_{0},\gamma (1)=x\},x\in \Gamma _{0}
\end{equation*}%
With this distance $\overline{\Gamma }_{0}$ is a complete metric space
determining the same topology as the differential structure. For $\Gamma _{0}
$ smooth enough, $d_{0}$ is a non-negative function in $C^{2}(\Gamma _{0})$
which we now extend. Let $V_{\rho }(\Gamma _{0})$ be a tubular neighborhood
of $\Gamma _{0}$ with the $C^{2}$ parametrization
\begin{equation*}
x=x(\xi ,s)=\xi +s\nu (\xi )\text{ for }\xi \in \Gamma _{0},s\in (-\rho
,\rho )
\end{equation*}%
and define
\begin{equation*}
\tilde{d}_{0}(x)=\tilde{d}_{0}(\xi ,s):=[s^{2}+d_{0}(\xi )^{2}]^{1/2}\text{
for every }x\in V_{\rho }(\Gamma _{0}).
\end{equation*}%
Now let $V_{0}:=V_{\rho }(\Gamma _{0})\backslash \{x_{0}\}$. The function $%
\tilde{d}_{0}$ is clearly in $C^{2}(V_{0})$ and we know that
\begin{equation}
\nu (\xi )\text{=-}\nabla \tilde{d}_{0}(\xi )\text{ for any }\xi \in \Gamma
_{0}  \label{normal equals gradient}
\end{equation}%
(see \cite[Theorem 8.5, Chapter 7]{delfour_shapes_2011}).
Furthermore, for any given $D$ precompact in $V_{0}$, $%
D\Subset V_{0}$, there exist positive constants $c_{d},C_{d}$ and $c_{\Delta
}$ depending on $\tilde{d}$ and $D$ such that
\begin{equation*}
c_{d}|x|\leqslant \tilde{d}_{0}(x)\leqslant C_{d}|x|\text{ and }\Delta
\tilde{d}_{0}(x)\leqslant c_{\Delta }\text{ for every }x\in D\Subset V_{0}.
\end{equation*}%
The second assertion is obvious since $\tilde{d}_{0}\in C^{2}(V_{0})$. For
the first simply let $m=\min_{x\in D}\tilde{d}_{0}(x)$, $M=\max_{x\in D}%
\tilde{d}_{0}(x)$, $l=\min_{x\in D}|x|$, $L=\max_{x\in D}|x|$, where we may
assume $l>0$ after a translation placing $x_{0}$ at the origin. Then it
suffices to take $c_{d}:=m/L$ and $C_{d}:=M/l$. Finally, using the extension
$\tilde{d}_{0}$ we may define for sufficiently small $\tau $ the balls
\begin{equation}
\tilde{B}_{0}(\tau ):=\{x\in V_{\rho }(\Gamma _{0}):\tilde{d}_{0}(x)<\tau \}.
\label{eq:def-extension-balls}
\end{equation}
Equipped with all this we may finally prove our main result:

\begin{theorem}
\label{th:coincidence-general}Let $\Omega $ be an open set in $\mathbbm{R}%
^{N}$ and suppose that the data $f,g$ and $\psi$ lead to a unique solution
$u\in H^{1}(\Omega )\cap L^{\infty }(\Omega )$.
Assume that there exist $\epsilon ,\delta \geq 0$, a subset $%
\Gamma _{\epsilon ,\delta } \subset \partial \Omega $ of class $C^{3}$, and a
tubular semi-neighborhood $V_{\rho }^{-}$ of $\Gamma _{\epsilon ,\delta }$
equipped with the intrinsic distance, for some $\rho > 0$ such that
\begin{equation}
\rho >0\text{ is large enough and }V_{\rho }^{-}\subset \overline{\Omega }
\end{equation}%
and such that if $\Psi \in H^{2}(V_{\rho }^{-})$ is a nonnegative
extension of $\psi $ to $V_{\rho }^{-}$ (i.e. $\Psi =\psi $ on $%
\Gamma _{\epsilon ,\delta }$) then we have that
\begin{equation*}
f+\Delta \Psi \leq -\epsilon \text{ on } V_{\rho}^{-},
\end{equation*}%
and
\begin{equation}
g-\partial _{\nu }\Psi \leqslant -\delta \text{ \ on \ }\Gamma _{\epsilon
,\delta }.
\end{equation}
If $\epsilon$ and $\delta$ are strictly positive, then one has the location
estimate $\Gamma _{\epsilon ,\delta }\subset I_{\psi }$, the coincidence
set of $u$.
\end{theorem}

\begin{proof}
As in Theorem \ref{th:coincidence-convex} it is enough to work with data $%
\widetilde{f}=f+\Delta \Psi $, $\widetilde{\psi }=0$ and $\tilde{g}%
:=g-\partial _{\nu }\Psi $. Let $x_{0}\in \Gamma _{\epsilon ,\delta }$ and $%
\rho _{\max }$ the maximal width of a tubular neighborhood around $\Gamma
_{\epsilon ,\delta }$. Define $R:=\min \{d(x_{0},\Gamma \backslash \Gamma
_{\epsilon ,\delta }),\rho _{\max }\}$, $D:=\Omega \cap \tilde{B}(x_{0},R)$,
$\partial _{1}D:=\partial D\backslash \Gamma $ and $\partial _{2}D:=\partial
D\cap \Gamma \subset \Gamma _{\epsilon ,\delta }$. Define now in $D$ the
function $\overline{u}(x)=c c_{\Delta }^{-1}\tilde{d}_{0}(x)$ for some $c>0$.
In $D$ we have
\begin{equation*}
-\Delta \overline{u}=-\frac{c}{c_{\Delta }}\Delta \tilde{d}_{0}(x)\geqslant
-\epsilon \geq \widetilde{f}.
\end{equation*}
Additionally, on $\partial _{1}D$ we have
\begin{equation*}
\overline{u}\geqslant \frac{c}{c_{\Delta }}R\geqslant M\geqslant \tilde{u},
\end{equation*}%
assuming that
\begin{equation*}
\frac{c}{c_{\Delta }}R\geqslant M.
\end{equation*}

Notice that this condition holds once we take $\rho >0$ (hence also $R$) large
enough (for some given $c>0$). Moreover, since  $\partial _{\nu }\widetilde{u}%
=\tilde{g}$ over $\partial _{2}D-I$, by using property (\ref{normal equals gradient})
we have
\begin{equation*}
\partial _{\nu }\overline{u}=c c_{\Delta }^{-1}\partial _{\nu }\tilde{d}%
_{0}=-c c_{\Delta }^{-1}\geqslant -\delta \geq \tilde{g}\text{ on }\partial
_{2}D),
\end{equation*}
once we take
\begin{equation*}
c c_{\Delta }^{-1}\leq \delta ,
\end{equation*}%
i.e.
\begin{equation*}
c\leq \min (\delta c_{\Delta },\epsilon ).
\end{equation*}
Moreover, since $\overline{u}=0$ in $\partial _{1}D\cup I$ we have $%
\overline{u}_{\alpha }\geqslant \tilde{u}_{\alpha }$ here. This yields $%
\tilde{u}\leqslant \overline{u}$ in $H^{1}(D)$ and the same inequality holds
for the traces, that is:
\begin{equation*}
0\leqslant \tilde{u}(x)\leqslant \frac{c}{c_{\Delta }}\tilde{d}_{0}(x),\text{
over }\Gamma \cap \tilde{B}(x_{0},R)\text{ in }H^{1/2}(\Gamma ).
\end{equation*}
Letting $x\rightarrow x_{0}$ we conclude that $\tilde{u}(x)=0$ a.e. in $%
\Gamma _{\epsilon ,\delta }$.
\end{proof}

\section{Final remarks and related work}
\label{sec:final-remarks}

The above results can be easily extended to the associated heat
equation with Signorini boundary conditions by using arguments similar to
those found in \cite{diaz_behaviour_1986}. Moreover, it is also possible to
extend them to Poisson's equation or to the heat equation
with dynamic boundary conditions as in \cite{diaz_aplicacion_1985} and 
\cite{diaz_special_2005} respectively. Notice that according to the equivalent
formulation of the fractional Laplacian operator (see, e.g.
\cite{caffarelli_extension_2007} and the multiple references given in
\cite{diaz_fractional_2018}), the case of dynamical Signorini boundary conditions 
for the Poisson and related linear equations corresponds to the usual obstacle problem
associated to the fractional Laplace operator.

The Signorini boundary conditions can be also formulated in terms of multivalued
nonlinear maximal monotone graphs (see \cite{Brezis:1972tt}). Some results
analyzing the set in which the solution vanishes on the boundary for other
different nonlinear boundary conditions was the main goal of the paper
\cite{diaz_free_2015}. See also \cite{davila_nonlinear_2005} for
the case of singular nonlinear boundary conditions.

The proof techniques used in the present scalar Signorini problem can also be
applied to the study of the contact region of one of the vectorial components
of the deformation field associated to some problems in linear elasticity (see, e.g., 
\cite{duvaut_inequalities_1976} and \cite{debenitodelgado_sobre_2012}).

A different class of problems to which the techniques of this paper can be
applied are the ones mentioned in the pioneering book
\cite{duvaut_inequalities_1976} concerning temperature control or reverse
osmosis membranes (see also the associated homogenization process in 
\cite{diaz_homogenization_2018}).

\bibliographystyle{apa}
\bibliography{remarks-signorini}
\label{page:e}

\bigskip

\bigskip

\noindent
\textsc{{M}iguel de {B}enito {D}elgado}\\
{U}niversity of {A}ugsburg / appliedAI @ UnternehmerTUM GmbH.\\
\href{mailto:m.debenito.d@gmail.com}{m.debenito.d@gmail.com}

\medskip

\noindent\textsc{{J}esus {I}ldefonso {D}íaz}\\
Instituto de Matemática Interdisciplinar, Universidad Complutense de Madrid.\\
Plaza de Ciencias 3, 28040 Madrid, Spain.\\
\href{mailto:jidiaz@ucm.es}{jidiaz@ucm.es}\\
\href{http://orcid.org/0000-0003-1730-9509}{http://orcid.org/0000-0003-1730-9509}\\

\noindent\textbf{Acknowledgements:} The research of the second author was partially
supported by the projects with ref. MTM2014-57113-P and MTM2017-85449-P of the
DGISPI (Spain).

\end{document}